\documentclass[a4paper,11pt,times,authoryears]{article}
\usepackage{authblk}
\usepackage{IEEEtrantools}
\usepackage[authoryear]{natbib}
\usepackage{bm}
\usepackage{amsmath}
\usepackage{amsthm}
\usepackage{mathtools}
\usepackage{amssymb,amsfonts}
\usepackage{graphicx}
\usepackage{mathrsfs}
\usepackage{multirow,multicol}
\usepackage{epstopdf}
\usepackage{enumerate}
\usepackage{booktabs}
\usepackage{adjustbox}
\usepackage{pdflscape}
\usepackage{courier}
\usepackage[colorlinks,citecolor=blue]{hyperref}
\usepackage{rotating}
\usepackage{fullpage}
\usepackage{setspace}
\usepackage{lscape}
\usepackage{color}
\usepackage{commath} 
\usepackage{subcaption}

\newtheorem{remark}{\textbf{Remark}}
\newtheorem{Theorem}{\textbf{Theorem}}
\newtheorem{Corollary}{\textbf{Corollary}}
\everydisplay{\abovedisplayshortskip\belowdisplayshortskip}%
\makeatletter \oddsidemargin  -.1in \evensidemargin -.1in

\numberwithin{equation}{section}
\usepackage{setspace}
\def\bk{{\mathbf k}} 
\def\bw{{\mathbf w}} 
\def\bx{{\mathbf x}}

\def\bz{\mathbf z}

\def\bI{\mathbf  I}

\def\bV{\mathbf  V}
\def\bW{\mathbf  W}
\def\bZ{\mathbf  Z}

\def\mC{\mathcal{C}}
\def\mE{\mathcal{E}}
\def\mH{\mathcal{H}}

\def\mK{\mathcal{K}}
\def\mL{\mathcal{L}}
\def\V{\mathcal{V}}
\def\Z{\mathcal{Z}}

\def\bC{\boldsymbol{\mC}}
\def\bK{\boldsymbol{\mK}}
\def\bZ{\boldsymbol{\Z}}

\def\bV{\boldsymbol{\V}}
\def\bv{\boldsymbol{\nu}}

\def\aa{\alpha}

\def\bb{\beta}
\def\a{\alpha_}
\def\b{\beta_}
\def\ahat{\widehat{\aa}_}
\def\bhat{\widehat{\bb}_}

\def\mr{m_r}
\def\nr{n_r}
\def\bzeros{\boldsymbol{0}}
\def\bones{\boldsymbol{1}}
\def\bpsi{\boldsymbol{\psi}}
\def\phat{\widehat{\bpsi}}
\def\pp{\partial}

\def\xi{x_{(i)}}
\def\no{\nonumber}
\def\ll{\left(}
\def\rr{\right)}

\title{\bf Inference for Two Lomax Populations Under Joint Type-II Censoring}

\author[1]{Yasin Asar\thanks{Corresponding Author}}
\author[2]{R.Arabi Belaghi}
\affil[1]{Department of Mathematics--Computer, Necmettin Erbakan University, Konya, Turkey, yasar@erbakan.edu.tr, yasinasar@hotmail.com}
\affil[2]{Department of Statistics, Faculty of Mathematical Sciences, University of Tabriz, Tabriz, Iran, rezaarabi11@gmail.com}
\begin{document}
\maketitle
\setcounter{page}{1}
\setcounter{equation}{0}
\setcounter{figure}{0}
\setcounter{table}{0}

\noindent{\bf{Abstract:}} Lomax distribution has been widely used in economics, business and actuarial sciences. Due to its importance, we consider the statistical inference of  this model under joint type-II censoring scenario. In order to estimate the parameters, we derive the Newton-Raphson(NR) procedure and we observe that most of the times in the simulation NR algorithm does not converge. Consequently, we make use of the expectation-maximization (EM) algorithm. Moreover, Bayesian estimations are also provided based on squared error, linear-exponential and generalized entropy loss functions together with the importance sampling method due to the structure of posterior density function. In the sequel, we perform a Monte Carlo simulation experiment to compare the performances of the listed methods. Mean squared error values, averages of estimated values as well as coverage probabilities and average interval lengths are considered to compare the performances of different methods. The approximate confidence intervals, bootstrap-p and bootstrap-t confidence intervals are computed for EM estimations. Also, Bayesian coverage probabilities and credible intervals are obtained. Finally, we consider the Bladder Cancer data to illustrate the applicability of the methods covered in the paper.
\vspace{0.2cm}

\vskip 6mm
\noindent {\bf AMS SUBJECT CLASSIFICATION:} 62F10; 62N01; 62N05.

\section{Introduction}
\label{Intro}
One sample problems have been well studied under different kinds of censoring schemes in the literature. When the experimenter wants to compare two populations that have been produced by two lines within the same facility, the joint censoring scheme has been developed so that the experimenter saves both time and money. 
More precisely, let us consider the two samples of sizes $m$ and $n$ selected from these two lines and put into a life testing experiment at once. Due to the purposes of experimental necessities, the experimenter can terminate the experiment as soon as a pre-determined number of failures occurs, say $r$, $1\leq r \leq N$ and $N=m+n$. 

We refer to the following studies considering the joint censoring in the literature:
\citet{basu} proposed a statistics which is the asymptotically most powerful rank test for censored data and it is equivalent to the Savage statistics. Locally most powerful rank test was considered by \cite{JM}. These studies and some others was reviewed in \citet{bala1995}. \citet{bala2008} studied the exact inference for jointly distributed two exponential distributions under the setting of type-II censoring. Extending this study, \citet{bala2010} discussed the exact inference for the two exponential distribution under joint type-II progressive censoring. \citet{parsi} developed the interval estimation for two Weibull distributions under joint Type-II progressive censoring. \citet{doost} studied the Bayesian estimation using squared error loss (SEL) and linear-exponential (LINEX) loss functions for a general class of distributions and discussed the Weibull distribution under jointly progressive type-II censoring scheme. \citet{shafay} considered the two exponential populations under the jointly type-II censored setting and developed the Bayesian inference for the unknown parameters using SEL, LINEX and general entropy loss (GEL) functions. Recently, \citet{kundu2017} studied the point and interval estimation of Weibull distribution under jointly progressive type-II censoring scheme using both likelihood and Bayesian estimations. Recently, \citet{volterman} developed inference for two exponential populations based on joint record values.

However, we consider the estimation of unknown parameters of two Lomax distributions under joint type-II censoring scheme in this study. Now, assume that $X_1, X_2, \ldots, X_m$ are the independent and identically distributed (i.i.d.) random variables representing the lifetimes of the first product and similarly $Y_1, Y_2, \ldots, Y_n$ are i.i.d. random variables from the other product having the following Lomax probability density functions (pdf) and cumulative distribution functions (cdf) respectively 
\begin{eqnarray}\label{pdf:Lomax}
&&f_i\left(x; \aa_i, \bb_i\right) = \aa_i \bb_i (1+\bb_i x)^{-\aa_i-1},~ x>0\\
\label{cdf:lomax}
&&F_i\left(x; \aa_i, \bb_i \right)=1-(1+\bb_i x)^{-\aa_i},~ x>0
\end{eqnarray}
where $\alpha_i, \beta_i>0$ for $i=1, 2$ are the shape and scale parameters, respectively. Let us also suppose that $W_1<W_2<\ldots<W_N$ represents the order statistics of the random variables $\{X_1, \ldots, X_m; Y_1, \ldots, Y_n\}$. The observed data can be denoted by $(\bV,\bW)$ where $\bW=\ll W_1, W_2, \ldots, W_r \rr$ and $\bV=\ll \V_1, \V_2, \ldots, \V_r \rr$ such that $\V_i=1$ if $W_i$ is from X-component and $\V_i=0$ if $W_i$ is from Y-component.

The rest of the paper is organized as follows: In Section \ref{MLE}, the maximum likelihood estimators (MLE) of the parameters are introduced by using Newton--Raphson (NR) algorithm and some theorems are developed regarding the exact distributions. Moreover, the expectation--maximization (EM) algorithm is also considered to obtain the MLEs, and the Fisher information matrix is obtained. In Section \ref{Bayes}, Bayes estimation for the unknown parameters of two Lomax distributions under the assumption of independent gamma priors using different loss functions such as SEL, LINEX and GEL functions. Moreover,  we performed a Monte Carlo simulation experiment to compare the efficiencies of these methods and discussed the results in Section \ref{sim}. A real data example is presented in Section \ref{RealData} to illustrate the findings of the study. Finally, some conclusive remarks are given in Section \ref{conclusion}.

\section{ Likelihood Inferences}\label{MLE}
In this section we consider the maximum likelihood inferences. We drive the likelihood equations and give the exact distribution of number of $\bV$.
Let $M_r=\sum_{i=1}^r\V_i$ and $M_r=\sum_{i=1}^r (1-\V_i)$ be the number of $X$ and $Y$ failures in $\bW$ respectively, such that $r=M_r+N_r$.
Then the likelihood function of the observable data $(\bV,\bW)$ is given by
\begin{IEEEeqnarray}{rCl}
\label{likelihood:data}
L(\aa_1, \bb_1,\aa_2, \bb_2, \bv,\bw) &=&\frac{m!n!}{(m-m_r)!(n-n_r)!} \prod_{i=1}^{r}f_1\ll w_i; \aa_1, \bb_1\rr ^{\nu_i}f_2\ll w_i; \aa_2, \bb_2\rr ^{1-\nu_i}\no\\
&&\times\left[1-F_1(w_r;\aa_1, \bb_1)\right]^{m-m_r}\left[1-F_2(w_r;\aa_2, \bb_2)\right]^{n-n_r} 
\end{IEEEeqnarray}
where $0<w_1<w_2<\ldots<w_r<\infty$.

Based on the observed data, the corresponding log-likelihood function without the additive constant can be expressed as
\begin{IEEEeqnarray}{rCl}
\label{log-likelihood:data}
l\left(\a1,\b1,\a2,\b2, \bv,\bw \right)&=&\mr\left[\ln(\a1)+\ln(\b1)\right]+\nr\left[\ln(\a2)+\ln(\b2)\right]-(\a1+1)\sum_{i=1}^r\nu_i\ln(1+\b1w_i)\no\\&&-\a1(m-\mr)\ln(1+\b1w_r)
-(\a2+1)\sum_{i=1}^r(1-\nu_i)\ln(1+\b2w_i)\no\\
&&-\a2(n-\nr)\ln(1+\b2w_r).
\end{IEEEeqnarray}
Taking the partial derivatives of Equation \eqref{log-likelihood:data} with respect to the parameters and equating them to zero, one can obtain the following normal equations respectively
\begin{IEEEeqnarray}{rCl}
\label{ll:a1}
\frac{\partial l}{\partial \a1} &=& \frac{\mr}{\a1}-\sum_{i=1}^r\nu_i\ln(1+\b1w_i)-(m-\mr)\ln(1+\b1w_r)=0\\
\label{ll:b1}
\frac{\partial l}{\partial \b1} &=&\frac{\mr}{\b1}-(\a1+1)\sum_{i=1}^r \frac{\nu_i w_i}{1+\b1w_i}-\frac{\a1(m-\mr)w_r}{1+\b1w_r}=0\\
\label{ll:a2}
\frac{\partial l}{\partial \a2} &=& \frac{\nr}{\a2}-\sum_{i=1}^r(1-\nu_i)\ln(1+\b2w_i)-(n-\nr)\ln(1+\b2w_r)=0\\
\label{ll:b2}
\frac{\partial l}{\partial \b2} &=&\frac{\nr}{\b2}-(\a2+1)\sum_{i=1}^r \frac{(1-\nu_i) w_i}{1+\b2w_i}-\frac{\a2(n-\nr)w_r}{1+\b2w_r}=0
\end{IEEEeqnarray}
Upon solving Equations \eqref{ll:a1} and \eqref{ll:a2}, we readily get the followings
\begin{IEEEeqnarray}{rCl}
\label{nr:a1}
\a1(\b1)&=& \left\{\frac{1}{\mr}\left[ \sum_{i=1}^r\nu_i\ln(1+\b1w_i)+(m-\mr)\ln(1+\b1w_r)\right]\right\}^{-1}\\
\label{nr:a2}
\a2(\b2)&=& \left\{\frac{1}{\nr}\left[ \sum_{i=1}^r(1-\nu_i)\ln(1+\b2w_i)+(n-\nr)\ln(1+\b2w_r)\right]\right\}^{-1}
\end{IEEEeqnarray}
and plugging in these into Equations \eqref{ll:b1} and \eqref{ll:b2}, we obtain the following one dimensional optimization problems respectively
\begin{IEEEeqnarray}{rCl}
\frac{\mr}{\b1}&=&(\a1+1)\sum_{i=1}^r \frac{\nu_i w_i}{1+\b1w_i}-\frac{\a1(\b1)(m-\mr)w_r}{1+\b1w_r}\\
\frac{\nr}{\b2}&=& (\a2(\b2)+1)\sum_{i=1}^r \frac{(1-\nu_i) w_i}{1+\b2w_i}-\frac{\a2(n-\nr)w_r}{1+\b2w_r}.
\end{IEEEeqnarray}
\begin{remark}
From Equations \eqref{nr:a1} and \eqref{nr:a2}, it is obvious that the MLEs of $(\a1, \b1)$ and $(\a2,\b2)$ do not exist when $\sum_{i=1}^r\V_i=0$ or $r$. Therefore, these MLEs are only conditional MLEs which are conditioned on $1 \leq \V_r \leq r-1$. 
\end{remark}
In the following theorem we give the exact distribution of $\bV$.
\begin{Theorem}\label{Thm:pmf:V}
When $\b1=\b2$, the joint probability mass function of $\bV$ is
\begin{eqnarray}\label{pmf:V}
P\ll \bV=\bv \rr = m^{(m_r)} n^{(n_r)}\a1^{\mr} \a2^{\nr}\prod_{i=1}^r\frac{1}{\a1(m-m_{i-1})+\a2(n-n_{i-1})}
\end{eqnarray}
such that $\mH=\lbrace \bv=(\nu_1, \nu_2, ..., \nu_r): \nu=0 ~or~ 1\rbrace$ where $m^{(t)}=m!/(m-t)!$, $m_{j-1}=\sum_{i=1}^{j-1}\nu_i$, $n_{j-1}=\sum_{i=1}^{j-1}(1-\nu_i)$ and $m_0\equiv n_0\equiv 0$. 
\end{Theorem}
\begin{proof}
When $\b1=\b2=\bb$, from Equation \eqref{likelihood:data}, the joint probability function of $(\bV,\bW)$ becomes
\begin{IEEEeqnarray}{rCl}
\label{joint:pdf}
f(\bv,\bw)&=&\frac{m!n!}{(m-m_r)!(n-n_r)!}\a1^{\mr} \a2^{\nr}\bb^r\ll1+\bb w_r \rr^{-\a1(m-\mr)-\a2(n-\nr)}\no\\
&&\times\prod_{i=1}^r \ll1+\bb w_i \rr^{-(\a1+1)\nu_i-(\a2+1)(1-\nu_i)}
,~0<w_1<w_2<\ldots<w_r<\infty.
\end{IEEEeqnarray}
After some algebra, we readily obtain the following
\begin{IEEEeqnarray*}{rCl}
f(\bv,\bw)&=&\frac{m!n!}{(m-m_r)!(n-n_r)!}\a1^{\mr} \a2^{\nr}\bb^r \ll1+\bb w_r \rr^{-\a1(m-m_{r-1})-\a2(n-n_{r-1})}\\
&&\times\prod_{i=1}^{r-1} \ll1+\bb w_i \rr^{-(\a1+1)\nu_i-(\a2+1)(1-\nu_i)}.
\end{IEEEeqnarray*}
Now, integrating the above function with respect to $w_r$ and continuing in a similar manner with $w_{r-1}$, $w_{r-2}$,..., $w_1$, one finally gets the desired probability mass function given in \eqref{pmf:V}. 
\end{proof}
\begin{Corollary}\label{cor:pmv:V}
It is directly seen from Theorem \ref{Thm:pmf:V} that the probability mass function of $M_r=\sum_{i=1}^r \V_i$ is as follows
\begin{eqnarray*}
P(M_r=i)=\underset{\bv \in \mH}{\sum \ldots \sum} m^{(m_r)} n^{(n_r)}\a1^{\mr} \a2^{\nr}\prod_{i=1}^r\frac{1}{\a1(m-m_{i-1})+\a2(n-n_{i-1})}
\end{eqnarray*}
for $\mH=\lbrace \bv=(\nu_1, \nu_2, ..., \nu_r): \nu_j=0 ~or~ 1,  \sum_{j=1}^r\nu_j=i\rbrace$ such that $i=1,2,...,r$.
\end{Corollary}
\begin{Corollary}
From Corollary \ref{cor:pmv:V}, it is readily seen that the following simplified equations hold
\begin{eqnarray*}
&&P(M_r=0)=P(\bV=\bzeros)=\frac{n!}{(n-r)!}\a2^r \prod_{i=1}^r\frac{1}{m\a1+\a2(n-i+1)},\\
&&P(M_r=r)=P(\bV=\bones)=\frac{m!}{(m-r)!}\a1^r \prod_{i=1}^r\frac{1}{\a1(m-i+1)+n\a2}.
\end{eqnarray*}
\end{Corollary}
\subsection{Expectation-Maximization Algorithm}

The EM algorithm proposed by \citet{em1977} can be used to obtain the MLEs of the parameters $\alpha_i$ and $\beta_i$, $i=1,2$. It is known that the EM algorithm converges more reliably than NR.
Since type-II joint censoring scheme may be considered as a problem of missing data (see \citet{Ng2002}), it is possible to apply EM algorithm to obtain the MLEs of the parameters.

Now, we denote the incomplete (censored or missing) data by $\ll \bK, \bZ\rr$ where $\bK=\ll \mK_1,..., \mK_{N-r}\rr$ and $\bZ=\left(Z_1, ..., Z_{N-r} \right)$ such that $\mK_i=1$ if the censored observation $\Z_i$ is in $X$ and $\mK_i=0$ if $\Z_i$ is in the sample $Y$. It is readily seen that $\sum_{i=1}^{N-r}\mK_i=m-\mr$ and $\sum_{i=1}^{N-r}(1-\mK_i)=n-\nr$.
Upon combining both the observed and missing data, we denote the complete data as $\bC =\ll \bV, \bW, \bK,\bZ \rr$. The corresponding likelihood equation of the complete data can be written as
\begin{IEEEeqnarray}{rCl}
\label{likelihood:complete}
L_{\bC}(\aa_1, \bb_1,\aa_2, \bb_2, \bv,\bw, \bk, \bz) &=& \prod_{i=1}^{r}f_1\ll w_i; \aa_1, \bb_1\rr ^{\nu_i}f_2\ll w_i; \aa_2, \bb_2\rr ^{1-\nu_i} \no\\
&&\times \prod_{j=1}^{N-r}f_1\ll z_j; \a1, \b1\rr ^{k_j}f_2\ll z_j; \a2, \b2\rr ^{1-k_j}.
\end{IEEEeqnarray}
Therefore, the log-likelihood equation can be easily obtained by taking the natural logarithm of Equation \eqref{likelihood:complete} as follows:
\begin{IEEEeqnarray}{rCl}
\label{log-like:complete}
l_{\bC}&=&m(\ln(\a1)+ln(\b1))+n(\ln(\a2)+\ln(\b2)) - (\a1+1)\left\{\sum_{i=1}^r\nu_i\ln(1+\b1w_i)+\sum_{j=1}^{N-r}k_j\ln(1+\b1z_j) \right\}\no\\
&&- (\a2+1)\left\{\sum_{i=1}^r(1-\nu_i)\ln(1+\b2w_i)+\sum_{j=1}^{N-r}(1-k_i)\ln(1+\b2z_i) \right\}.
\end{IEEEeqnarray}
Based on the complete sample, the MLEs of the parameters $\a1, \b1, \a2$ and $\b2$ can be computed by taking the partial derivatives of Equation \eqref{log-like:complete} with respect to these parameters respectively and equating them to zero as follows:
\begin{eqnarray}
\label{diff:a1:comp}
&&\frac{\partial l_{\bC}}{\partial \a1} = \frac{m}{\a1}-\sum_{i=1}^r\nu_i\ln(1+\b1w_i)-\sum_{j=1}^{N-r}k_j\ln(1+\b1z_j)=0\\
\label{diff:b1:comp}
&&\frac{\partial l_{\bC}}{\partial \b1} = \frac{m}{\b1}- (\a1+1)\left\{\sum_{i=1}^r\frac{\nu_iw_i}{1+\b1w_i}+\sum_{j=1}^{N-r}\frac{k_jz_j}{1+\b1z_j }\right\}=0 \\
\label{diff:a2:comp}
&&\frac{\partial l_{\bC}}{\partial \a2} = \frac{n}{\a2}-\sum_{i=1}^r(1-\nu_i)\ln(1+\b2w_i)-\sum_{j=1}^{N-r}(1-k_j)\ln(1+\b2z_j)=0\\
\label{diff:b2:comp}
&&\frac{\partial l_{\bC}}{\partial \b2} = \frac{m}{\b2}- (\a2+1)\left\{\sum_{i=1}^r\frac{(1-\nu_i)w_i}{1+\b2w_i}+\sum_{j=1}^{N-r}\frac{(1-k_j) z_j}{1+\b2 z_j}\right\}=0. 
\end{eqnarray}
Now, in the E-step of the EM algorithm, we need the conditional expectations of the normal equations. Thus, any function of the random variable $\bZ$ should be replaced by its expectation. Therefore, Equations \eqref{diff:a1:comp}-\eqref{diff:b2:comp} become, respectively 
\begin{eqnarray*}
\label{cond:expect:a1}
&&E\left(\frac{\partial l_{\bC}}{\partial \a1} \Big|~ w_i \right)= \frac{m}{\a1}-\sum_{i=1}^r\nu_i\ln(1+\b1w_i)-\sum_{j=1}^{N-r}k_j E\left[\ln(1+\b1 Z_j)~\Big|~ \Z_j>w_r\right]=0\\
\label{cond:expect:b1}
&&E\left(\frac{\partial l_{\bC}}{\partial \b1} \Big|~ w_i \right) = \frac{m}{\b1}- (\a1+1)\left\{\sum_{i=1}^r\frac{\nu_iw_i}{1+\b1w_i}+\sum_{j=1}^{N-r}k_j E\left[\frac{Z_j}{1+\b1 Z_j } ~\Big|~ \Z_j>w_r\right]\right\}=0\\
\label{cond:expect:a2}
&&E\left(\frac{\partial l_{\bC}}{\partial \a2} \Big|~ w_i \right) = \frac{n}{\a2}-\sum_{i=1}^r(1-\nu_i)\ln(1+\b2w_i)-\sum_{j=1}^{N-r}(1-k_j) E\left[\ln(1+\b2 Z_j)~\Big|~ \Z_j>w_r\right]=0\\
\label{cond:expect:b2}
&&E\left(\frac{\partial l_{\bC}}{\partial \b2} \Big|~ w_i \right) =\frac{n}{\b2}- (\a2+1)\left\{\sum_{i=1}^r\frac{(1-\nu_i)w_i}{1+\b2w_i}+\sum_{j=1}^{N-r}(1-k_j)E\left[\frac{Z_j}{1+\b2 Z_j } ~\Big|~ \Z_j>w_r\right]\right\}=0.
\end{eqnarray*}
Following \citet{Ng2002} and \citet{singh}, we see that the conditional distribution of $\bZ$ is truncated Lomax distribution from the left at $w_r$ and it has the following probability density function
\begin{eqnarray}\label{cond:pdf}
f_{\bK,\bZ|\bV,\bW}(z_{j}|w_r)=\left\{\frac{f_1(z_{j};\a1,\b1)}{1-F_1(w_r;\a1,\b1)}\right\}^{k_j}\left\{\frac{f_2(z_{j};\a2,\b2)}{1-F_2(w_r;\a2,\b2)}\right\}^{1-k_j}, \Z_{j}>w_r.
\end{eqnarray}
Note that this conditional pdf has two components such that if $k_j=1$ then it reduces to the first component, and if $k_j=0$ then it reduces to the second component.
Using the conditional pdf given in \eqref{cond:pdf}, the following expectations can be easily obtained as
\begin{eqnarray*}
&&\mE_{1}\left(w_r;\aa, \bb \right)= E\left[\ln(1+\bb \Z_j) \Big| \Z_j >w_r \right] = \ln(1+\bb w_r)+\frac{1}{\aa}\\
&&\mE_{2}\left(w_r;\aa, \bb \right) = E\left[\frac{\Z_j}{1+\bb \Z_j} \Big| \Z_j >w_r \right] = \frac{1+(\aa+1 )w_r}{\bb(\aa+1)(1+\bb w_r)}
\end{eqnarray*}
see \citet{Helu} and \citet{reza} for more details.

Upon updating the missing data with the expectations above in the E-step, the log-likelihood function is maximized in the M-step at the $(k+1)$th stage by estimating $\a1^{k+1}$ and $\a2^{k+1}$ using
\begin{eqnarray}
&&\ahat1^{k+1} = \left\{\frac{1}{m} \sum_{i=1}^{r}\nu_i\ln(1+\bhat1^kw_i)+\sum_{j=1}^{N-r}k_j\mE_1(w_r;\ahat1^k,\bhat1^k)\right\}^{-1},\no\\
\label{EM:a2:update}
&&\ahat2^{k+1} = \left\{\frac{1}{n} \sum_{i=1}^{r}(1-\nu_i)\ln(1+\bhat2^k w_i)+\sum_{j=1}^{N-r}(1-k_j)\mE_1(w_r;\ahat2^k,\bhat2^k)\right\}^{-1}\no.
\end{eqnarray}
Once $\ahat1^{k+1}$ and $\ahat2^{k+1}$ are computed, one can readily obtain $\bhat1^{k+1}$ and $\bhat2^{k+1}$ respectively as follows
\begin{eqnarray}
&&\bhat1^{k+1}=\left\{\frac{1}{m} (\ahat1^{k+1}+1)\left[ \sum_{i=1}^{r}\frac{\nu_i w_i}{1+\bhat1^{k} w_i}+\sum_{j=1}^{N-r}k_j\mE_2(w_r;\ahat1^{k+1},\bhat1^k)\right]\right\}^{-1},\no\\
\label{EM:b2:update}
&&
\bhat2^{k+1}=\left\{\frac{1}{n} (\ahat2^{k+1}+1)\left[ \sum_{i=1}^{r}\frac{(1-\nu_i) w_i}{1+\b2^{k} w_i}+\sum_{j=1}^{N-r}(1-k_j)\mE_2(w_r;\ahat2^{k+1},\bhat2^k)\right]\right\}^{-1}\no.
\end{eqnarray}
The EM estimates of the parameters $\left( \a1, \b1, \a2, \b2 \right)$ can be computed by this iterative procedure until convergence is reached. 

\subsection{Fisher Information Matrix}
In this subsection, by making use of the idea of missing information principle proposed by \citet{Louis1982}, we can  obtain the observed Fisher
information matrix. \citet{Louis1982} suggested the following relation
\begin{eqnarray}\label{info:matrix}
\bI_{\bV,\bW}\ll\bpsi\rr = \bI_{\bC}\ll\bpsi\rr - \bI_{\bK,\bZ \mid \bV,\bW}\ll\bpsi\rr
\end{eqnarray}
where $\bpsi=\ll\a1,\b1,\a2,\b2,\rr'$, $\bI_{\bV,\bW}\ll\bpsi\rr$, $\bI_{\bC}\ll\bpsi\rr$ and
$\bI_{\bK,\bZ \mid \bV,\bW}\ll\bpsi\rr$ are the observed, complete and missing information matrices respectively. Now, the information matrix of a complete data set following the Lomax distribution can be obtained as
\begin{eqnarray}\label{comp:inf:mat}
\bI_{\bC}\ll\bpsi\rr = - E\ll\frac{\partial^2 \ln\mL}{\partial \bpsi^2}\rr =\left[
  \begin{array}{cccc}
   \frac{m}{\a1^2} & \frac{m}{\b1(\a1+1)} & 0& 0\\ \\
   \frac{m}{\b1(\a1+1)} & \frac{m\a1}{\b1^2(\a1+2)} & 0 & 0\\ \\
   0 & 0 & \frac{n}{\a2^2} & \frac{n}{\b2(\a2+1)}\\ \\
   0 & 0 &\frac{n}{\b2(\a2+1)} & \frac{n\a2}{\b2^2(\a2+2)}
  \end{array}
\right]
\end{eqnarray}
where $\ln \mL \ll \bpsi \rr = m\ln\a1+m\ln\b1 +n\ln\a2+n\ln\b2-
(\a1+1)\sum_{i=1}^{m}\ln(1+\b1x_i)-
(\a2+1)\sum_{j=1}^n \ln(1+\b2y_j)$ is the corresponding log-likelihood equation.

Moreover, the missing
information matrix $\bI_{\bK,\bZ \mid \bV,\bW}\ll\bpsi\rr$ is given by
\begin{equation}\label{missing:inf:mat}
\bI^{(j)}_{\bK,\bZ \mid \bV,\bW}\ll\bpsi\rr = -E\ll \frac{\pp^2\ln f_{\bK,\bZ|\bV,\bW}(z_{j}|w_r, z_j>w_r)}{\pp \bpsi^2}\rr
\end{equation}
where the minus expected values of the second partial derivatives of $\ln f_{\bK,\bZ|\bV,\bW}(z_{j}|w_r, z_j>w_r)$ are computed as follows
\begin{eqnarray*}
&& -E\ll\frac{\pp^2\ln f_{\bK,\bZ|\bV,\bW}}{\pp\a1^2}\rr=\frac{1}{\a1^2},~~-E\ll\frac{\pp^2\ln f_{\bK,\bZ|\bV,\bW}}{\pp\a1\b1}\rr=\frac{1}{\b1(\a1+1)(1+\b1w_r)},\\
&& -E\ll\frac{\pp^2\ln f_{\bK,\bZ|\bV,\bW}}{\pp\b1^2}\rr=\frac{\a1}{\b1^2(\a1+2)(1+\b1w_r)^2},\\
&& -E\ll\frac{\pp^2\ln f_{\bK,\bZ|\bV,\bW}}{\pp\a2^2}\rr=\frac{1}{\a2^2},~~-E\ll\frac{\pp^2\ln f_{\bK,\bZ|\bV,\bW}}{\pp\a2\b2}\rr=\frac{1}{\b2(\a2+1)(1+\b2w_r)},\\
&& -E\ll\frac{\pp^2\ln f_{\bK,\bZ|\bV,\bW}}{\pp\b2^2}\rr=\frac{\a2}{\b2^2(\a2+2)(1+\b2w_r)^2}.
\end{eqnarray*}
Notice that all the remaining entries of the missing information matrix are equal to zero. Hence, the asymptotic variance-covariance matrix of the parameter vector $\widehat{\bpsi}$ can be readily computed by $\ll\bI_{\bV,\bW}\ll\widehat{\bpsi}\rr\rr^{-1}$ such that $\widehat{\bpsi}$ is obtained using the EM estimates of the parameters.
Therefore, an approximate $(1 -\aa)100\%$ confidence interval of $\bpsi_i$ can be constructed by $$\ll \widehat{\bpsi}_i-z_{\aa/2}\sqrt{var\ll \widehat{\bpsi}_i\rr},~\widehat{\bpsi}_i+z_{\aa/2}\sqrt{var\ll \widehat{\bpsi}_i\rr}  \rr$$
where $var\ll \widehat{\bpsi}_i\rr$ is the $i^{th}$ diagonal element of $\ll\bI_{\bV,\bW}\ll\widehat{\bpsi}\rr\rr^{-1}$ for $i=1,2,3,4$.
\subsection{Bootstrap Confidence Intervals}
In this subsection, we also propose to use the bootstrapping method to construct confidence intervals for EM estimations. We use the Boot-p and Boot-t algorithms proposed by \citet{Efron} and \citet{Hall} respectively. The algorithms are given as follows:
\begin{itemize}
    \item[i)] \textbf{Boot-p algorithm:} In this method, one can construct confidence intervals using the $100(\alpha/2)$th and $100(1-\alpha/2)$th quantiles of the empirical bootstrapped sample of $\phat^*$. Namely,
    \begin{itemize}
    \item[1)] Compute the EM estimation $\phat^*$ of $\bpsi^*$ based on the current jointly censored sample $\ll \bv, \bw \rr$.
    \item[2)] Compute the bootstrapped estimate $\phat^*$ by re-sampling from the original data with replacement, say $\ll \bv^*, \bw^* \rr$. 
    \item[3)] Repeat this process $D$ times to obtain the sorted estimations in ascending order as $$\phat^*_{(1)},\phat^*_{(2)},\ldots,\phat^*_{(D)}$$
    \item[4)] Finally, a $100(1- \alpha)\%$ Boot-p confidence interval can be written as $\ll \phat^*_{\ll D\aa/2\rr}, \phat^*_{\ll D(1-\aa/2)\rr} \rr$.
    \end{itemize}
    
    \item[i)]\textbf{Boot-t algorithm:} After generating the bootstrap sample as given above, do the following steps:
    \begin{itemize}
    \item[1)] Compute the following t-statistics $T\ll\phat^* \rr= \ll \phat^*-\phat \rr/se(\phat^*)$ such that $se(\phat^*)$ is the bootstrapped standard error.
    \item[2)] Repeating this step $D$ times and sorting the bootstrap sample, obtain
    $$T\ll\phat^*\rr_{(1)},T\ll\phat^*\rr_{(2)},\ldots,T\ll\phat^*\rr_{(D)}$$
    
    \item[3)] Finally, a $100(1- \alpha)\%$ Boot-t confidence interval can be written as $$\ll \phat+T\ll\phat^*\rr_{\ll D\aa/2\rr}, \phat+T\ll\phat^*\rr_{\ll D(1-\aa/2)\rr} \rr$$
    \end{itemize}
\end{itemize}
\section{Bayesian Inferences}\label{Bayes}

In this section, we consider the Bayesian estimation for the parameters of the joint Lomax distribution under the assumption that all the parameters have the independent gamma priors such that $\a1 \sim G(a_1,b_1)$,
$\b1 \sim G(c_1, d_1)$, $\a2 \sim G(a_2,b_2)$ and
$\b2 \sim G(c_2, d_2)$. More precisely,  the prior functions are given as
\begin{IEEEeqnarray*}{rCl}
\pi\ll \a1 \rr &\propto& \a1^{a_1-1~}e^{-b_1 \a1}, ~\pi\ll \a2 \rr \propto \a2^{a_2-1~}e^{-b_2 \a2},\\
\pi\ll \b1 \rr &\propto& \a1^{c_1-1~}e^{-d_1 \b1}, ~\pi\ll \b2 \rr \propto \b2^{c_2-1~}e^{-d_2 \b2}.
\end{IEEEeqnarray*}

Therefore, using the likelihood function given in Equation \eqref{likelihood:data}, the posterior joint density function can be obtained as follows
\begin{IEEEeqnarray}{rCl}
\label{indep:posterior}
\pi\ll \bpsi \mid data \rr
&\propto& \a1^{\mr+a_1-1}\b1^{\mr+c_1-1}\a2^{\nr+a_2-1}\b2^{\nr+c_2-1}
\no \\ 
&&\times{\rm exp}\left\{ -\a1 \ll b_1+(m-\mr)\ln(1+\b1 w_r)+\sum_{i=1}^r\nu_i\ln(1+\b1 w_i) \rr -d_1\b1\right\}
\no \\ 
&& \times
{\rm exp}\left\{ -\a2 \ll b_2+(n-\nr)\ln(1+\b2 w_r)+\sum_{i=1}^r(1-\nu_i)\ln(1+\b2 w_i) \rr  -d_2\b2\right\}\no \\ 
&& \times {\rm exp}\left\{-\sum_{i=1}^r\nu_i\ln(1+\b1 w_i) -\sum_{i=1}^r(1-\nu_i)\ln(1+\b2 w_i)\right\}.
\end{IEEEeqnarray}

In this paper, three different loss functions are considered. One of them is the most commonly used squared error loss function (SEL) which is defined as follows:
\begin{eqnarray*}
L_S\ll\widehat{t}(\psi), t(\psi)\rr = \ll\widehat{t}(\psi)-t(\psi)\rr^2
\end{eqnarray*}
where $\widehat{t}(\psi)$ is an estimator of $t(\psi)$. It is known that SEL, being a symmetric loss function, gives equal weights to both underestimation and overestimation. This is a drawback when the over estimation and underestimation have not same importance. To overcome this problem, linear-exponential (LINEX) loss function introduced by
\citet{linex} as follows
\begin{eqnarray*}\label{linex}
L_L\ll\widehat{t}(\psi), t(\psi)\rr =e^{\nu\ll\widehat{t}(\psi)-t(\psi)\rr}-\nu\ll\widehat{t}(\psi)-t(\psi)\rr-1, \;\nu\neq0.
\end{eqnarray*}
The LINEX loss function is an asymmetric, convex function whose shape is determined by the value of $\nu$. Determining the degree of asymmetry, the negative values of $\nu$ result in overestimation and positive values of $\nu$ result in underestimation.
Therefore, the Bayes estimate of $t(\psi)$ under the LINEX loss function is given by
\begin{eqnarray*}\label{llbays}
\widehat{t}_{L}(\psi)=-\frac{1}{\nu}\ln\left[E_{t}\ll e^{-\nu t(\psi)}\mid\bm{x} \rr \right]= -\frac{1}{\nu}\ln\left[\int_{0}^{\infty}\int_{0}^{\infty}e^{-\nu t(\psi)}\pi(\aa, \beta \mid \bm{x})d\aa \, d\beta\right].
\end{eqnarray*}
Finally, the general entropy asymmetric  loss (GEL) function is also considered and it is given by
\begin{eqnarray*}\label{gen}
L_{GEL}\ll\widehat{t}(\psi), t(\psi)\rr
=\ll\frac{\widehat{t}(\psi)}{t(\psi)}\rr^{\kappa}-\kappa\ln\ll\frac{\widehat{t}(\psi)}{t(\psi)}\rr -1,\: \kappa\neq 0.
\end{eqnarray*}
where $\kappa$ is a parameter determining the shape of the function and representing the degree of symmetry. $\kappa>0$ corresponds to the overestimation and $\kappa<0$ corresponds to underestimation. The Bayes estimator under GEL function is given by
\begin{eqnarray*}
\widehat{t}_{GEL}(\psi)=\left[E_{t}\ll  t(\psi)^{-\kappa}\mid\bm{x} \rr \right]^{-1/\kappa}= \left[\int_{0}^{\infty}\int_{0}^{\infty}t(\psi)^{-\kappa}\pi(\aa, \beta \mid \bm{x})d\aa \, d\beta\right]^{-1/\kappa}.
\end{eqnarray*}

\subsection{Importance Sampling}

Notice that the posterior density function given in Equation \eqref{indep:posterior} can also be written in the following form 
\begin{IEEEeqnarray}{rCl}\label{post}
\pi\ll \bpsi \mid data \rr
&\propto& G_{\b1}\ll \mr+c_1, d_1\rr \times G_{\a1|\b1}\ll \mr+a_1, K_1\rr \times G_{\b2}\ll \nr+c_2, d_2\rr \times G_{\a2|\b2}\ll \nr+a_2, K_2\rr \no\\
&&\times \frac{{\rm exp}\left\{-\sum_{i=1}^r \nu_i\ln(1+\b1w_i) -\sum_{i=1}^r (1-\nu_i)\ln(1+\b2w_i)\right\}}{K_1^{m_r+a_1}K_2^{n_r+a_2}} 
\end{IEEEeqnarray}
where 
$$K_1=b_1+(m-\mr)\ln(1+\b1 w_r)+\sum_{i=1}^r \nu_i\ln(1+\b1w_i),$$
$$K_2=b_2+(n-\nr)\ln(1+\b2 w_r)+\sum_{i=1}^r (1-\nu_i)\ln(1+\b2w_i),$$
$G_{\b1}$ and $G_{\b2}$ denote the distributions of $\b1$ and $\b2$ respectively and $G_{\a1|\b1}$ and $G_{\a2|\b2}$ represent the distributions of $\a1$ and $\a2$ given $\b1$ and $\b2$ respectively.
Now, we can consider the following steps to produce samples from the posterior density given in \eqref{post}
\begin{itemize}
    \item[(1)] 
    Generate $\b1$ and $\b2$ using $G_{\b1}\ll \mr+c_1, d_1\rr$ and  $G_{\b2}\ll \nr+c_2, d_2\rr$, respectively
    \item[(2)] 
    Given $\b1$ and $\b2$ from previous step, generate $\a1$ and $\a2$ using $G_{\a1|\b1}\ll \mr+a_1, K_1\rr$ and $G_{\a2|\b2}\ll \nr+a_2, K_2\rr$, respectively
    \item[(3)]
    Repeat the steps (1) and (2) to compute $\bpsi_i=(\a{1i},\b{1i},\a{2i},\b{2i})$ for $i=1,2,...,T.$
\end{itemize}
After generating $T$ samples, the Bayes estimates under SEL, LINEX and GEL loss functions can be computed as follows
\begin{IEEEeqnarray*}{rCl}
\widehat{t}_{\rm SEL}(\bpsi)&=& \frac{\sum_{j=1}^T t(\bpsi_j)\mH(\bpsi_j)}{\sum_{j=1}^M\mH(\bpsi_j)}, \\
\widehat{t}_{\rm LINEX}(\bpsi) &=& -\frac{1}{\nu}\ln\ll \frac{\sum_{j=1}^T {\rm exp}\ll -\nu t(\bpsi_j)\rr\mH(\bpsi_j)}{\sum_{j=1}^T \mH(\bpsi_j)}\rr, \\
\widehat{t}_{\rm GEL}(\bpsi)&=&\ll \frac{\sum_{j=1}^T  t(\bpsi_j)^{-\kappa} \mH(\bpsi_j)}{\sum_{j=1}^T \mH(\bpsi_j)}  \rr^{-1/\kappa}
\end{IEEEeqnarray*}
where
\begin{eqnarray*}
\mH(\bpsi)={\rm exp}\ll -\sum_{i=1}^r \nu_i\ln(1+\b{1j}w_i)-\sum_{i=1}^r (1-\nu_i)\ln(1+\b{2j}w_i) \rr \Big/ \ll K_1^{m_r+a_1}K_2^{n_r+a_2}\rr.
\end{eqnarray*}
In order to construct a Bayesian credible interval, using the idea of 
\citet{chen-shao}, we consider the posterior density $\pi(\eta | \bx)$ of a parameter $\eta$. Assume that $$\theta^{(p)}={\rm inf}\left\{\theta:\Pi(\theta|\bx) \geq p; 0<p<1 \right\}$$ represents the pth quantile of the distribution is where $\Pi(\eta|\bx)$ denotes the posterior cumulative distribution function of $\eta$. Now, given the value of $\eta^*$, one can define a simulation consistent estimator of $\Pi(\eta^* | \bx)$ as
\begin{eqnarray*}
\Pi(\eta^*|\bx)= \frac{1}{M}\sum_{i=1}^M I(\eta \leq \eta^*)
\end{eqnarray*}
where $I(\eta \leq \eta^*)$ is an indicator function. Then, $\Pi(\eta^*|\bx)$ is estimated by
\begin{eqnarray*}
\widehat{\Pi}(\eta^*|\bx) = 
\begin{cases}
      0 & \text{if} ~\eta^*<\eta_{(1)}\\
      \sum_{j=1}^i \gamma_j & \text{if} ~\eta_{(i)}<\eta^*<\eta_{(i+1)}\\
      1 & \text{if} ~\eta_{(M)}
    \end{cases}  
\end{eqnarray*}
where $\gamma_j=1/M$ and $\theta_{(j)}$ is the jth ordered value of $\theta_{j}$. Thus,
$\theta^{(p)}$ can be approximated by the following
\begin{eqnarray*}
\theta^{(p)}=\begin{cases}
      \theta_{(1)} & \text{if}~ p=0\\
      \theta_{(j)} & \text{if}~\sum_{j=1}^{i-1} \gamma_j<p<\sum_{j=1}^i \gamma_j
\end{cases}
\end{eqnarray*}
Now, the $100(1-p)\%$ confidence intervals can be defined as $\ll\widehat{\eta}^{j/s}, \widehat{\eta}^{(j+[(1-p)s])/s} \rr$, $j=1, 2,..., s-[(1-p)s]$ in which  $[.]$ denotes the greatest integer function. Then, the interval having the shortest length can be taken as the credible interval of $\eta$.
\section{Monte Carlo Simulation Experiments}\label{sim}

In this section, we conduct a Monte Carlo simulation to evaluate the performance of EM and Bayes estimation methods. We consider the following different values for the two populations as $m=20, 40, 80$ and $n=20, 40, 80$. The size of censored sample is taken as $r=10, 20, 30, 40 , 80.$ The real values of the parameters are chosen to be $(\a1=2, \b1=3, \a2=3, \b2=5)$. For each setting, we compute the MLEs using EM algorithm. Bayes estimates under SEL, LINEX with $\nu=-0.5$ and $\nu=0.5$, GEL with $\kappa=-0.5$ and $\kappa=0.5$ are also computed by generating a size of $10^4$ importance sampling procedure together with the following informative prior values $(a_1=4, b_1=2, a_2=6, b_2=2)$, these values are chosen so that the prior means are equal to the real parameter values. However, the same argument does not work for the hyper-parameters $c_1, d_1, c_2, d_2$. Because the scale parameters of the distributions of $\bb_1$ and $\bb_2$ depend on $d_1$ and $d_2$, respectively. Fixing the values of $d_1$ and $d_2$ and computing $c_1$ and $c_2$ and then using these values for all situations, we observe that increasing the value of $r$ affects the performance of Bayes estimates dramatically. Therefore, we propose to choose $c_1, d_1, c_2, d_2$ for each scenario accordingly as given in Table \ref{hyper}.
\begin{table}[ht]
\centering
\caption{The hyper-parameter values in the simulation}
\label{hyper}
\begin{tabular}{cccccccccccc}
 & \multicolumn{3}{c}{$m=n=20$}&&\multicolumn{3}{c}{$m=n=40$}&&\multicolumn{3}{c}{$m=n=80$} \\
\cmidrule{2-4}\cmidrule{6-8}\cmidrule{10-12}
    $r$ & 10 & 20 & 30 && 10 & 20 & 40  && 20 & 40  & 80 \\ \hline
  $c_1$ & 3  & 12 & 48 && 3  & 12 & 60  && 15 & 60  & 240 \\ 
  $d_1$ & 1  & 4  & 16 && 1  & 4  & 20  && 5  & 20  & 80 \\ 
  $c_2$ & 15 & 20 & 60 && 15 & 60 & 100 && 75 & 225 & 850 \\ 
  $d_2$ & 3  & 4  & 12 && 3  & 12 & 20  && 15 & 45  & 170 \\ 
   \hline
\end{tabular}
\end{table}
Moreover, the 95\% approximate confidence intervals and bootstrapped confidence intervals using Boot-t and Boot-p methods for MLE and Bayes credible intervals are obtained. 
Totally, $10^4$ repetitions are carried out and average values (Avg), mean squared errors (MSE), confidence/credible interval lengths (IL) and coverage probabilities (CP) are obtained for the purpose of comparison. MSEs of the estimators are computed as follows
\begin{eqnarray*}
{\rm MSE}\ll\widehat{\theta}\rr=\frac{1}{10^4}\sum_{i=1}^{10^4}\ll\widehat{\theta}_i-\theta \rr^2
\end{eqnarray*}
where $\widehat{\theta}_i$ is EM and Bayes estimators under SEL loss function in the ith replication. On the other hand, the MSEs of Bayes estimators under LINEX and GEL loss functions are computed respectively by
\begin{IEEEeqnarray*}{rCl}
{\rm MSE_{LINEX}}\ll\widehat{\theta}\rr&=&\frac{1}{10^4}\sum_{i=1}^{10^4}\ll e^{\nu\ll\widehat{\theta}_i-\theta\rr}-\nu\ll\widehat{\theta}_i-\theta\rr-1\rr, \\
{\rm MSE_{GEL}}\ll \widehat{\theta}\rr&=&\frac{1}{10^4}\sum_{i=1}^{10^4}\ll\ll\frac{\widehat{\theta}_i}{\theta}\rr^{\kappa}-\kappa\ln\ll\frac{\widehat{\theta}_i}{\theta}\rr -1\rr.
\end{IEEEeqnarray*}
All of the computations are performed using the R Statistical Program \citep{R2018}. 
We tabulate the results of the simulation in Tables \ref{Res:EM}-\ref{Res:CP-IL}.
In Table \ref{Res:EM}, we summarize the average values (Avg) and corresponding MSEs of EM estimates based on different values of $m, n$ and $r$. We observe from this table that all of the estimates of the parameters have satisfactory performances in terms of both Avg and MSE. It is worthy to note that even with small values of $r$ the MSEs of the estimators are  quite small. Generally. Increasing the value of $r$ makes a decrease in the values of MSEs. Tables \ref{Res:SEL}-\ref{Res:GEL} show the Bayes estimates of the parameters based on SEL, LINEX and GEL functions. Based on Table \ref{Res:SEL}, it is observed that the MSEs of all of the Bayes estimates are smaller than the MSEs of EM estimates. Also, we have seen that the MSEs of both LINEX and GEL estimates are smaller than the MSEs of SEL. Thus, we can conclude that the Bayes estimates are preferable to EM estimates in terms of having smaller MSEs. Table \ref{Res:CP-IL} presents the 95\% coverage probabilities (CP) (with nominal of 95\%) and corresponding average interval lengths (IL). 

Generally, Bayes CPs are higher than CPs of EM method. Increasing the values of $r$ affects Bayes CPs and ILs positively and Bayes CPs become larger than CPs of EM. Although, Boot-t and Boot-p methods provide reasonably high CPs, they are always less than that of Bayes and EM. If we consider the ILs, then Boot-t and Boot-p methods have quite small ILS than the other methods. Finally, Bayes ILs are always less than approximate ILs of EM.


\section{Real Data Example}\label{RealData}
In this section, we analyze the bladder cancer data which was given in \citet{LeeWang} and also analyzed by \citet{rady} using the Lomax distribution. This data consists of remission times (in months) of a sample of 128 bladder cancer patients. To illustrate the findings of the paper, we divided the data into two samples by randomly sampling $40$ observations and considering these observations as the $X$ sample, and the remaining $88$ observations are taken as the $Y$ sample, see Table \ref{Data}.

\begin{table}[!ht]
\centering
\caption{Bladder cancer data divided into two samples}
\label{Data}
\begin{tabular}{rrrrrrrrr}
Data: X \\ 
  \hline
1 & 6.940 & 17.140 & 0.510 & 2.640 & 4.340 & 20.280 & 2.691 & 2.260 \\ 
  2 & 17.120 & 0.810 & 2.540 & 46.120 & 5.320 & 5.090 & 9.220 & 3.640 \\ 
  3 & 10.060 & 0.400 & 32.150 & 7.390 & 13.290 & 8.260 & 6.540 & 3.250 \\ 
  4 & 7.870 & 2.460 & 3.880 & 8.650 & 43.010 & 2.830 & 2.690 & 15.960 \\ 
  5 & 7.320 & 7.590 & 3.310 & 10.750 & 3.700 & 5.060 & 19.360 & 34.260 \\ 
  \hline
Data: Y \\ 
  \hline
1 & 0.080 & 6.970 & 5.170 & 4.180 & 4.260 & 5.620 & 5.850 & 12.030 \\ 
  2 & 2.090 & 9.020 & 7.280 & 5.340 & 5.410 & 11.640 & 11.980 & 2.020 \\ 
  3 & 3.480 & 3.570 & 9.740 & 10.660 & 7.630 & 17.360 & 19.130 & 3.360 \\ 
  4 & 4.870 & 7.090 & 14.760 & 36.660 & 1.260 & 1.400 & 1.760 & 6.760 \\ 
  5 & 8.660 & 13.800 & 26.310 & 1.050 & 4.330 & 3.020 & 4.500 & 12.070 \\ 
  6 & 13.110 & 25.740 & 2.620 & 4.230 & 5.490 & 5.710 & 6.250 & 21.730 \\ 
  7 & 23.630 & 0.500 & 3.820 & 5.410 & 7.660 & 7.930 & 8.370 & 2.070 \\ 
  8 & 0.200 & 7.260 & 14.770 & 7.620 & 11.250 & 11.790 & 12.020 & 3.360 \\ 
  9 & 2.230 & 9.470 & 10.340 & 16.620 & 79.050 & 18.100 & 2.020 & 6.930 \\ 
  10 & 3.520 & 14.240 & 14.830 & 1.190 & 1.350 & 1.460 & 4.510 & 12.630 \\ 
  11 & 4.980 & 25.820 & 0.900 & 2.750 & 2.870 & 4.400 & 8.530 & 22.690 \\ 
   \hline
\end{tabular}
\end{table}
Then, we fit Lomax distribution to each sample and report the results in Table \ref{KS}. We provided the Kolmogorov-Smirnov test statistic values (D) and the corresponding p-values, saying that the data fit the Lomax distribution with the parameters given in Table \ref{KS}.
\begin{table}[!ht]
\centering
\caption{MLEs and Kolmogorov-Smirnov test results for data}
\label{KS}
\begin{tabular}{ccccc}
Data & $\widehat{\alpha}_i$ & $\widehat{\beta}_i$ & D & p-value \\
  \hline
  X & 10.9770 & 0.0098 & 0.1396 & 0.3814 \\ 
  Y & 4.0034  & 0.0321 & 0.1133 & 0.2089 \\ 
   \hline
\end{tabular}
\end{table}
We used the R package \texttt{fitdistrplus}, created by \citet{fitdistr}, uses the \texttt{optim} function to obtain the MLEs. Then MLEs are used as initial values in the EM algorithm. We also consider the following hyper-parameter values as the informative priors for the Bayesian estimators $a_1=110, b_1=10, c_1=2, d_1=200$, $a_2=40, b_2=10, c_2=1, d_2=300$ by simply equating the means of the priors to the corresponding MLEs. We report the estimated values and the corresponding confidence/credible intervals in Tables \ref{EST:data}-\ref{CICP:data}.

According to Table \ref{EST:data}, it can be concluded that the Bayes estimates based on different loss functions are very close to each other. Further, as $r$ increases, the EM estimates get close to Bayes estimates, especially when $r=40$. From Table \ref{CICP:data}, it is seen that the confidence interval of EM estimates are very wider than those based on Bayes estimates due to the high variance of EM estimates. Moreover, we can say that the lower bounds of EM confidence interval are always zero. Overall, we prefer to use Bayes confidence interval because of their small length. 
\begin{table}[htbp]
\centering
\caption{Estimated values of EM and Bayes methods for different values of $r$}
\label{EST:data}
\begin{tabular}{lcrrrr}
&$r$ & $\ahat1$ & $\bhat1$ & $\ahat2$ & $\bhat2$ \\ 
  \hline
EM & 10 & 8.3684 & 0.0077 & 2.8704 & 0.0239 \\ 
   & 20 & 7.1149 & 0.0068 & 3.2875 & 0.0271 \\ 
   & 30 & 10.0378 & 0.0092 & 3.2306 & 0.0274 \\ 
   & 40 & 10.9600 & 0.0097 & 3.3398 & 0.0283 \\ 
  \hline
SEL & 10 & 10.8670 & 0.0072 & 4.0941 & 0.0105 \\ 
 & 20 & 11.0626 & 0.0062 & 4.1864 & 0.0206 \\ 
 & 30 & 10.0080 & 0.0185 & 3.2243 & 0.0401 \\ 
 & 40 & 10.5093 & 0.0314 & 2.8159 & 0.0485 \\ 
  \hline
LINEX        & 10 & 11.1497 & 0.0072 & 4.1939 & 0.0105 \\ 
$(\nu=-0.5)$ & 20 & 11.1089 & 0.0062 & 4.2132 & 0.0206 \\ 
             & 30 & 10.0745 & 0.0185 & 3.3130 & 0.0401 \\ 
             & 40 & 10.5267 & 0.0314 & 2.8238 & 0.0485 \\ 
  \hline
LINEX       & 10 & 10.5947 & 0.0072 & 3.9979 & 0.0104 \\ 
$(\nu=0.5)$ & 20 & 11.0035 & 0.0062 & 4.1596 & 0.0206 \\ 
            & 30 & 9.9377 & 0.0185 & 3.1457 & 0.0400 \\ 
            & 40 & 10.4876 & 0.0314 & 2.8075 & 0.0485 \\ 
  \hline
GEL             & 10 & 10.8416 & 0.0069 & 4.0700 & 0.0101 \\ 
$(\kappa=-0.5)$ & 20 & 11.0577 & 0.0061 & 4.1800 & 0.0205 \\ 
                & 30 & 10.0010 & 0.0181 & 3.1993 & 0.0397 \\ 
                & 40 & 10.5074 & 0.0313 & 2.8129 & 0.0483 \\ 
  \hline
GEL            & 10 & 10.7904 & 0.0064 & 4.0216 & 0.0093 \\ 
$(\kappa=0.5)$ & 20 & 11.0476 & 0.0061 & 4.1668 & 0.0204 \\ 
               & 30 & 9.9870 & 0.0175 & 3.1512 & 0.0389 \\ 
               & 40 & 10.5035 & 0.0311 & 2.8066 & 0.0480 \\ 
   \hline
\end{tabular}
\end{table}
\begin{table}[!ht]
\centering
\caption{Confidence and credible intervals of EM and Bayes methods for different values of $r$} \label{CICP:data}
\begin{tabular}{clrrrrrrrrrrr}
  &&\multicolumn{2}{c}{$\a1$}&&\multicolumn{2}{c}{$\b1$}&&\multicolumn{2}{c}{$\a2$}&&\multicolumn{2}{c}{$\b2$}\\
\cmidrule{3-4}\cmidrule{6-7}\cmidrule{9-10}\cmidrule{12-13}
 r&& \multicolumn{1}{c}{L} & \multicolumn{1}{c}{U} && \multicolumn{1}{c}{L} & \multicolumn{1}{c}{U} && \multicolumn{1}{c}{L} & \multicolumn{1}{c}{U} && \multicolumn{1}{c}{L} & \multicolumn{1}{c}{U}\\  
  \hline
10&  ACI    & 0.0000 & 42.8955 && 0.0000 & 0.0398 && 0.0000 & 10.7252 && 0.0000 & 0.0903 \\ 
  & Boots.p & 2.8310 & 15.0191 && 0.0027 & 0.0130 && 2.2316 & 3.7967 && 0.0193 & 0.0304 \\ 
  & Boots.t & 1.4827 & 15.2541 && 0.0020 & 0.0134 && 1.9351 & 3.8057 && 0.0172 & 0.0306 \\ 
  & BAYES   & 8.1390 & 12.2926 && 0.0082 & 0.0511 && 2.6008 & 4.9677 && 0.0115 & 0.0482 \\ 
   \hline
20& ACI    & 0.0000 & 34.5814 && 0.0000 & 0.0335 && 0.0000 & 26.8106 && 0.0000 & 0.2268 \\ 
  &Boots.p & 3.8603 & 11.0070 && 0.0039 & 0.0098 && 2.4228 & 4.0463 && 0.0216 & 0.0321 \\ 
  &Boots.t & 3.4365 & 10.7933 && 0.0037 & 0.0099 && 2.3563 & 4.2187 && 0.0216 & 0.0327 \\ 
  &BAYES   & 7.1263 & 11.4647 && 0.0110 & 0.0582 && 1.9914 & 3.9400 && 0.0330 & 0.0867 \\ 
   \hline
30& ACI    & 0.0000 & 97.7040 && 0.0000 & 0.0903 && 0.0000 & 11.6072 && 0.0000 & 0.1011 \\ 
  &Boots.p & 7.2152 & 13.0827 && 0.0071 & 0.0109 && 2.6431 & 3.9937 && 0.0240 & 0.0316 \\ 
  &Boots.t & 7.0117 & 13.0638 && 0.0072 & 0.0112 && 2.5186 & 3.9426 && 0.0233 & 0.0315 \\ 
  &BAYES   & 5.7038 & 9.6871  && 0.0309 & 0.0982 && 1.6830 & 3.3537 && 0.0434 & 0.1029 \\ 
   \hline
 40 & ACI  & 0.0000 & 80.8883 && 0.0000 & 0.0729 && 0.0000 & 15.4225 && 0.0000 & 0.1355 \\ 
  &Boots.p & 7.9621 & 13.8563 && 0.0078 & 0.0111 && 2.7864 & 4.0055  && 0.0253 & 0.0317 \\ 
  &Boots.t & 7.8754 & 14.0447 && 0.0080 & 0.0115 && 2.7059 & 3.9737  && 0.0249 & 0.0316 \\ 
  &BAYES   & 4.8903 & 8.4258  && 0.0457 & 0.1235 && 1.4474 & 2.8138  && 0.0595 & 0.1269 \\ 
   \hline
\end{tabular}
\end{table}

\section{Conclusive Remarks}\label{conclusion}
In this paper, we discussed the estimation problem of joint type-II censored data from two Lomax populations. Although we obtained MLEs via the Newton-Raphson (NR) method in a theoretical framework, we observed that NR method is not stable and does not converge most of the time in our simulation studies. Therefore, we made use of EM algorithm to estimate the parameters and construct the asymptotic confidence intervals as well as bootstrap-p and bootstrap-t confidence intervals. EM algorithm always converges in the simulation. We concluded that the asymptotic variance of EM turns out to be large. However, Boot-p and Boot-t confidence intervals are very much narrower. Moreover, we also consider the Bayesian estimation using independent gamma priors based on squared error loss, LINEX loss and generalized entropy loss functions. Since there are four parameters of two Lomax populations, we had eight hyper-parameters of the prior distributions. In order to obtain better performance, it is needed to select the hyper-parameters suitably due to the dependence of scale parameters of posterior distributions of $\bb_1$ and $\bb_2$ only on the hyper-parameter $d_1$ and $d_2$ respectively. Upon choosing suitable values of these parameters, Bayesian methods produced better performance than EM algorithm. Due to the complexity of posterior distribution, importance sampling method was employed to generate data from posterior. We also computed the average lengths and credible intervals of Bayesian method. We conducted a Monte Carlo simulation to compare the listed methods. According to the results, Bayes estimators had better performance in terms of MSE. 

As future studies, it is possible to consider the estimation on two Lomax population under the jointly progressive censoring or further extending the study to more than two populations under different censoring schemes.
\vskip 1cm
\noindent\textbf{Acknowledgements.} This paper was written while Dr. Yasin Asar visited McMaster University and he was supported by The Scientific and Technological Research Council of Turkey (TUBITAK), BIDEB-2219 Postdoctoral Research Program, Project No: 1059B191700537. We are grateful to Prof. N. Balakrishnan for his guidance and comments.

\section{Appendix}
\begin{table}[!ht]
\centering
\caption{Average Values (Avg) and Mean Squared Errors (MSE) of EM estimations.}
\centering \label{Res:EM}
\begin{tabular}{ccccccc}
$(m,n)$& $r$ & & $\ahat1$ & $\bhat1$ & $\ahat2$ & $\bhat2$ \\
  \hline
$(20,20)$ & 10 & Avg & 2.128 & 3.136 & 3.174 & 5.159 \\  
             &&  MSE & 0.621 & 1.054 & 1.015 & 1.894 \\ 
           &20&  Avg & 2.123 & 3.100 & 3.215 & 5.084 \\ 
             &&  MSE & 0.494 & 0.691 & 1.069 & 1.610 \\ 
           &30&  Avg & 2.217 & 3.001 & 3.370 & 4.972 \\ 
             &&  MSE & 0.734 & 0.567 & 1.511 & 1.856 \\
  \hline
$(40,40)$&20&  Avg & 2.090 & 3.111 & 3.053 & 5.083 \\ 
           &&  MSE & 0.480 & 0.959 & 0.730 & 1.852 \\ 
         &30&  Avg & 2.067 & 3.095 & 3.098 & 5.122 \\ 
           &&  MSE & 0.376 & 0.754 & 0.711 & 1.646 \\ 
         &40&  Avg & 2.085 & 3.070 & 3.166 & 5.082 \\ 
           &&  MSE & 0.335 & 0.586 & 0.751 & 1.553 \\
  \hline
$(80,80)$&20&  Avg & 2.023 & 3.036 & 3.009 & 5.048 \\ 
           &&  MSE & 0.318 & 0.707 & 0.588 & 1.650 \\ 
         &40&  Avg & 2.043 & 3.060 & 3.062 & 5.092 \\ 
           &&  MSE & 0.279 & 0.607 & 0.569 & 1.522 \\ 
         &80&  Avg & 2.065 & 3.071 & 3.125 & 5.089 \\ 
           &&  MSE & 0.264 & 0.541 & 0.629 & 1.513 \\ 
 \hline
\end{tabular}
\end{table}
 
\begin{table}[!ht]
\centering
\caption{Average Values (Avg) and Mean Squared Errors (MSE) of Bayes estimations under SEL function.}
\centering
\label{Res:SEL}
\begin{tabular}{ccccccc}
$(m,n)$& $r$ & & $\ahat1$ & $\bhat1$ & $\ahat2$ & $\bhat2$ \\
  \hline
$(20,20)$ & 10 & Avg & 2.069 & 3.163 & 3.102 & 5.080 \\ 
             &&  MSE & 0.124 & 0.438 & 0.302 & 0.131 \\ 
           &20&  Avg & 2.081 & 3.060 & 3.078 & 5.027 \\ 
             &&  MSE & 0.193 & 0.061 & 0.321 & 0.018 \\ 
           &30&  Avg & 2.090 & 3.017 & 3.083 & 5.017 \\ 
             &&  MSE & 0.194 & 0.008 & 0.305 & 0.011 \\
  \hline
$(40,40)$&20&  Avg & 2.061 & 3.150 & 3.098 & 5.082 \\ 
           &&  MSE & 0.123 & 0.458 & 0.295 & 0.136 \\ 
         &30&  Avg & 2.062 & 3.052 & 3.087 & 5.029 \\ 
           &&  MSE & 0.187 & 0.057 & 0.318 & 0.018 \\ 
         &40&  Avg & 2.046 & 3.017 & 3.061 & 5.014 \\ 
           &&  MSE & 0.177 & 0.007 & 0.244 & 0.006 \\
  \hline
$(80,80)$&20&  Avg & 2.048 & 3.027 & 3.076 & 5.016 \\ 
           &&  MSE & 0.197 & 0.037 & 0.314 & 0.010 \\ 
         &40&  Avg & 2.039 & 3.013 & 3.060 & 5.011 \\ 
           &&  MSE & 0.179 & 0.006 & 0.239 & 0.004 \\ 
         &80&  Avg & 2.028 & 3.007 & 3.031 & 5.006 \\ 
           &&  MSE & 0.114 & 0.002 & 0.147 & 0.001 \\ 
 \hline
\end{tabular}
\end{table}
\begin{table}[!ht]
\centering
\caption{Average Values (Avg) and Mean Squared Errors (MSE) of Bayes estimations under LINEX function.}
\label{Res:LINEX}
\begin{tabular}{cccccccccccc}
&&&  \multicolumn{4}{c}{$\nu=-0.5$}&&\multicolumn{4}{c}{$\nu=0.5$}\\
\cmidrule{4-7}\cmidrule{9-12}
$(m,n)$& $r$ & & $\ahat1$ & $\bhat1$ & $\ahat2$ & $\bhat2$ && $\ahat1$ & $\bhat1$ & $\ahat2$ & $\bhat2$ \\ 
  \hline
$(20,20)$&10&  Avg & 2.285 & 3.907 & 3.348 & 5.471 && 1.903 & 2.742 & 2.902 & 4.770 \\ 
           &&  MSE & 0.034 & 0.256 & 0.076 & 0.052 && 0.013 & 0.044 & 0.030 & 0.019 \\ 
         &20&  Avg & 2.212 & 3.229 & 3.225 & 5.124 && 1.970 & 2.920 & 2.948 & 4.937 \\ 
           &&  MSE & 0.040 & 0.015 & 0.063 & 0.004 && 0.020 & 0.008 & 0.035 & 0.003 \\ 
         &30&  Avg & 2.170 & 3.060 & 3.197 & 5.075 && 2.018 & 2.976 & 2.979 & 4.962 \\ 
           &&  MSE & 0.035 & 0.001 & 0.056 & 0.002 && 0.022 & 0.001 & 0.034 & 0.002 \\
\hline
$(40,40)$&20&  Avg & 2.279 & 3.884 & 3.343 & 5.469 && 1.894 & 2.731 & 2.899 & 4.774 \\ 
           &&  MSE & 0.033 & 0.254 & 0.074 & 0.052 && 0.013 & 0.047 & 0.030 & 0.020 \\  
         &30&  Avg & 2.202 & 3.220 & 3.232 & 5.126 && 1.945 & 2.910 & 2.958 & 4.940 \\ 
           &&  MSE & 0.037 & 0.014 & 0.063 & 0.004 && 0.020 & 0.008 & 0.035 & 0.003 \\ 
         &40&  Avg & 2.114 & 3.052 & 3.142 & 5.043 && 1.984 & 2.985 & 2.985 & 4.987 \\ 
           &&  MSE & 0.029 & 0.001 & 0.040 & 0.001 && 0.021 & 0.001 & 0.029 & 0.001 \\ 
\hline
$(80,80)$&20&  Avg & 2.187 & 3.165 & 3.218 & 5.095 && 1.931 & 2.907 & 2.951 & 4.942 \\ 
           &&  MSE & 0.038 & 0.008 & 0.060 & 0.002 && 0.022 & 0.005 & 0.035 & 0.002 \\ 
         &40&  Avg & 2.113 & 3.047 & 3.138 & 5.037 && 1.971 & 2.979 & 2.987 & 4.986 \\ 
           &&  MSE & 0.029 & 0.001 & 0.039 & 0.001 && 0.021 & 0.001 & 0.028 & 0.001 \\ 
         &80&  Avg & 2.063 & 3.015 & 3.073 & 5.013 && 1.995 & 2.998 & 2.991 & 4.999 \\ 
           &&  MSE & 0.017 & 0.000 & 0.021 & 0.000 && 0.014 & 0.000 & 0.018 & 0.000 \\ 
   \hline
\end{tabular}
\end{table}
\begin{table}[!ht]
\centering 
\caption{Average Values (Avg) and Mean Squared Errors (MSE) of Bayes estimations under GEL function.}
\label{Res:GEL}
\begin{tabular}{cccccccccccc}
&&&  \multicolumn{4}{c}{$\kappa=-0.5$}&&\multicolumn{4}{c}{$\kappa=0.5$}\\
\cmidrule{4-7}\cmidrule{9-12}
$(m,n)$& $r$ & & $\ahat1$ & $\bhat1$ & $\ahat2$ & $\bhat2$ && $\ahat1$ & $\bhat1$ & $\ahat2$ & $\bhat2$ \\ 
  \hline
$(20,20)$&10&  Avg & 1.982 & 3.006 & 3.033 & 5.013 && 1.808 & 2.699 & 2.896 & 4.881 \\ 
           &&  MSE & 0.004 & 0.006 & 0.004 & 0.001 && 0.006 & 0.009 & 0.004 & 0.001 \\ 
         &20&  Avg & 2.026 & 3.012 & 3.035 & 5.008 && 1.915 & 2.916 & 2.948 & 4.972 \\
           &&  MSE & 0.006 & 0.001 & 0.004 & 0.000 && 0.006 & 0.001 & 0.004 & 0.000 \\  
         &30&  Avg & 2.056 & 3.003 & 3.049 & 5.006 && 1.986 & 2.975 & 2.980 & 4.983 \\ 
           &&  MSE & 0.005 & 0.000 & 0.004 & 0.000 && 0.005 & 0.000 & 0.004 & 0.000 \\
\hline
$(40,40)$&20&  Avg & 1.973 & 2.993 & 3.030 & 5.016 && 1.797 & 2.685 & 2.893 & 4.885 \\ 
           &&  MSE & 0.004 & 0.007 & 0.004 & 0.001 && 0.006 & 0.010 & 0.004 & 0.001 \\ 
         &30&  Avg & 2.003 & 3.002 & 3.044 & 5.011 && 1.884 & 2.905 & 2.958 & 4.974 \\ 
           &&  MSE & 0.006 & 0.001 & 0.004 & 0.000 && 0.007 & 0.001 & 0.004 & 0.000 \\ 
         &40&  Avg & 2.015 & 3.006 & 3.036 & 5.009 && 1.953 & 2.985 & 2.986 & 4.998 \\ 
           &&  MSE & 0.005 & 0.000 & 0.003 & 0.000 && 0.005 & 0.000 & 0.003 & 0.000 \\ 
\hline
$(80,80)$&20&  Avg & 1.988 & 2.985 & 3.034 & 5.001 && 1.869 & 2.902 & 2.950 & 4.971 \\ 
           &&  MSE & 0.007 & 0.001 & 0.004 & 0.000 && 0.008 & 0.001 & 0.004 & 0.000 \\ 
         &40&  Avg & 2.005 & 3.001 & 3.036 & 5.006 && 1.937 & 2.979 & 2.987 & 4.996 \\ 
           &&  MSE & 0.006 & 0.000 & 0.003 & 0.000 && 0.006 & 0.000 & 0.003 & 0.000 \\
         &80&  Avg & 2.012 & 3.004 & 3.018 & 5.005 && 1.979 & 2.998 & 2.991 & 5.002 \\ 
           &&  MSE & 0.003 & 0.000 & 0.002 & 0.000 && 0.004 & 0.000 & 0.002 & 0.000 \\ 
   \hline
\end{tabular}
\end{table}
\begin{table}[!ht]
\centering
\caption{The estimated 95\% coverage probabilities (CP) and average lengths of confidence intervals (IL) of EM and Bayes estimations.}
\label{Res:CP-IL}
\begin{tabular}{cclcccccrrrr}
&&&  \multicolumn{4}{c}{CP}&&\multicolumn{4}{c}{IL}\\
\cmidrule{4-7}\cmidrule{9-12}
  $(m,n)$& $r$ && $\ahat1$ & $\bhat1$ & $\ahat2$ & $\bhat2$ && $\ahat1$ & $\bhat1$ & $\ahat2$ & $\bhat2$ \\  
  \hline
$(20,20)$ & 10 &ACI & 97.66 & 97.93 & 93.95 & 96.11 && 4.85 & 7.57 & 5.94 & 10.45 \\ 
          && Boot.t & 84.75 & 78.52 & 90.00 & 74.72 && 2.41 & 2.66 & 3.44 & 3.39 \\
          && Boot.p & 87.56 & 84.02 & 85.84 & 68.25 && 2.51 & 2.86 & 3.01 & 3.07 \\
          &&  BAYES & 99.71 & 85.54 & 96.99 & 90.47 && 2.58 & 9.52 & 3.05 & 6.08 \\
           &20& ACI & 93.03 & 95.97 & 85.53 & 91.30 && 3.88 & 6.42 & 4.66 & 8.94 \\
          && Boot.t & 88.67 & 67.62 & 88.64 & 48.85 && 2.36 & 1.73 & 3.17 & 2.21 \\ 
          && Boot.p & 86.10 & 70.88 & 84.79 & 44.46 && 2.13 & 1.72 & 2.98 & 1.80 \\
          &&  BAYES & 95.62 & 60.34 & 96.53 & 98.40 && 2.07 & 4.26 & 2.58 & 2.78 \\
           &30& ACI & 84.81 & 92.55 & 80.24 & 88.81 && 3.11 & 5.55 & 4.25 & 8.30 \\
          && Boot.t & 97.48 & 70.63 & 93.81 & 78.19 && 5.41 & 1.89 & 4.11 & 15.68 \\
          && Boot.p & 85.21 & 58.96 & 85.80 & 60.66 && 6.61 & 1.91 & 4.23 & 10.15 \\
           && BAYES & 95.68 & 95.20 & 95.75 & 99.80 && 1.84 & 1.90 & 2.33 & 2.12 \\
 \hline
$(40,40)$ &20& ACI & 97.54 & 96.77 & 94.48 & 95.32 && 4.90 & 7.42 & 5.96 & 10.16 \\
         && Boot.t & 79.71 & 74.33 & 79.47 & 69.57 && 1.89 & 2.32 & 2.38 & 3.17 \\
         && Boot.p & 79.75 & 80.29 & 76.80 & 66.97 && 2.04 & 2.64 & 2.18 & 2.93 \\
         &&  BAYES & 99.92 & 86.36 & 96.38 & 86.29 && 2.64 & 9.39 & 3.02 & 6.10 \\
          &30& ACI & 94.74 & 96.10 & 86.44 & 88.64 && 4.02 & 6.34 & 4.71 & 8.29 \\
         && Boot.t & 81.21 & 69.99 & 78.92 & 56.15 && 1.66 & 1.89 & 2.16 & 2.20 \\
         && Boot.p & 80.53 & 73.38 & 75.88 & 53.60 && 1.67 & 1.99 & 2.02 & 2.09 \\
          && BAYES & 97.99 & 75.05 & 95.89 & 91.51 && 2.18 & 4.17 & 2.52 & 2.80 \\
          &40& ACI & 85.72 & 89.49 & 75.68 & 82.72 && 3.02 & 4.97 & 3.58 & 6.75 \\
         && Boot.t & 79.20 & 54.71 & 77.70 & 30.07 && 1.43 & 1.19 & 2.04 & 1.21 \\
         && Boot.p & 77.92 & 54.38 & 74.42 & 28.82 && 1.40 & 1.20 & 1.97 & 1.18 \\
          && BAYES & 94.96 & 82.13 & 95.65 & 90.64 && 1.71 & 1.69 & 2.02 & 1.47 \\
 \hline
$(80,80)$ &20&  ACI & 95.66 & 96.24 & 87.11 & 88.47 && 4.11 & 6.29 & 4.68 & 8.11 \\
          && Boot.t & 69.35 & 62.03 & 61.80 & 48.90 && 1.28 & 1.64 & 1.50 & 2.01 \\ 
          && Boot.p & 68.07 & 65.55 & 61.86 & 50.06 && 1.28 & 1.72 & 1.43 & 1.92 \\
          && BAYES  & 98.76 & 90.20 & 96.00 & 98.01 && 2.27 & 3.57 & 2.52 & 2.46 \\
          &40& ACI  & 87.93 & 90.51 & 76.98 & 79.37 && 3.20 & 4.98 & 3.52 & 6.28 \\
          && Boot.t & 69.67 & 59.05 & 63.91 & 39.39 && 1.14 & 1.33 & 1.44 & 1.48 \\
          && Boot.p & 68.77 & 58.75 & 63.34 & 39.19 && 1.12 & 1.35 & 1.38 & 1.43 \\
          && BAYES  & 96.55 & 91.27 & 95.08 & 86.39 && 1.80 & 1.67 & 1.97 & 1.38 \\
          &80& ACI  & 75.64 & 80.48 & 65.80 & 72.46 && 2.32 & 3.81 & 2.64 & 5.01 \\
          && Boot.t & 65.10 & 39.72 & 60.76 & 19.99 && 0.98 & 0.84 & 1.38 & 0.83 \\
          && Boot.p & 64.12 & 38.50 & 59.75 & 19.41 && 0.96 & 0.84 & 1.35 & 0.82 \\
           && BAYES & 94.61 & 82.52 & 95.07 & 98.39 && 1.32 & 0.80 & 1.52 & 0.69 \\
  \hline
\end{tabular}
\end{table}

\end{document}